\documentclass[11pt, reqno]{amsart}
\usepackage{amssymb}
\usepackage{eucal}
\usepackage{amsmath}
\usepackage{amscd}
\usepackage[dvips]{color}
\usepackage{multicol}
\usepackage[all]{xy}           
\usepackage{graphicx}
\usepackage{color}
\usepackage{colordvi}
\usepackage{xspace}
\usepackage{tikz}

\usepackage[active]{srcltx} 

\begin{document}
\newcommand {\emptycomment}[1]{} 

\baselineskip=14pt
\newcommand{\nc}{\newcommand}
\newcommand{\delete}[1]{}
\nc{\mfootnote}[1]{\footnote{#1}} 
\nc{\todo}[1]{\tred{To do:} #1}

\nc{\mlabel}[1]{\label{#1}}  
\nc{\mcite}[1]{\cite{#1}}  
\nc{\mref}[1]{\ref{#1}}  
\nc{\mbibitem}[1]{\bibitem{#1}} 

\delete{
\nc{\mlabel}[1]{\label{#1}  
{\hfill \hspace{1cm}{\bf{{\ }\hfill(#1)}}}}
\nc{\mcite}[1]{\cite{#1}{{\bf{{\ }(#1)}}}}  
\nc{\mref}[1]{\ref{#1}{{\bf{{\ }(#1)}}}}  
\nc{\mbibitem}[1]{\bibitem[\bf #1]{#1}} 
}

\newtheorem{thm}{Theorem}[section]
\newtheorem{lem}[thm]{Lemma}
\newtheorem{cor}[thm]{Corollary}
\newtheorem{pro}[thm]{Proposition}
\newtheorem{ex}[thm]{Example}
\newtheorem{rmk}[thm]{Remark}
\newtheorem{defi}[thm]{Definition}
\newtheorem{pdef}[thm]{Proposition-Definition}
\newtheorem{condition}[thm]{Condition}

\renewcommand{\labelenumi}{{\rm(\alph{enumi})}}
\renewcommand{\theenumi}{\alph{enumi}}

\nc{\tred}[1]{\textcolor{red}{#1}}
\nc{\tblue}[1]{\textcolor{blue}{#1}}
\nc{\tgreen}[1]{\textcolor{green}{#1}}
\nc{\tpurple}[1]{\textcolor{purple}{#1}}
\nc{\btred}[1]{\textcolor{red}{\bf #1}}
\nc{\btblue}[1]{\textcolor{blue}{\bf #1}}
\nc{\btgreen}[1]{\textcolor{green}{\bf #1}}
\nc{\btpurple}[1]{\textcolor{purple}{\bf #1}}

\nc{\cm}[1]{\textcolor{red}{Chengming:#1}}
\nc{\yy}[1]{\textcolor{blue}{Yanyong: #1}}
\nc{\lit}[2]{\textcolor{blue}{#1}{}} 
\nc{\yh}[1]{\textcolor{green}{Yunhe: #1}}


\nc{\twovec}[2]{\left(\begin{array}{c} #1 \\ #2\end{array} \right )}
\nc{\threevec}[3]{\left(\begin{array}{c} #1 \\ #2 \\ #3 \end{array}\right )}
\nc{\twomatrix}[4]{\left(\begin{array}{cc} #1 & #2\\ #3 & #4 \end{array} \right)}
\nc{\threematrix}[9]{{\left(\begin{matrix} #1 & #2 & #3\\ #4 & #5 & #6 \\ #7 & #8 & #9 \end{matrix} \right)}}
\nc{\twodet}[4]{\left|\begin{array}{cc} #1 & #2\\ #3 & #4 \end{array} \right|}

\nc{\rk}{\mathrm{r}}
\newcommand{\g}{\mathfrak g}
\newcommand{\h}{\mathfrak h}
\newcommand{\pf}{\noindent{$Proof$.}\ }
\newcommand{\frkg}{\mathfrak g}
\newcommand{\frkh}{\mathfrak h}
\newcommand{\Id}{\rm{Id}}
\newcommand{\gl}{\mathfrak {gl}}
\newcommand{\ad}{\mathrm{ad}}
\newcommand{\add}{\frka\frkd}
\newcommand{\frka}{\mathfrak a}
\newcommand{\frkb}{\mathfrak b}
\newcommand{\frkc}{\mathfrak c}
\newcommand{\frkd}{\mathfrak d}
\newcommand {\comment}[1]{{\marginpar{*}\scriptsize\textbf{Comments:} #1}}

\nc{\gensp}{V} 
\nc{\relsp}{\Lambda} 
\nc{\leafsp}{X}    
\nc{\treesp}{\overline{\calt}} 

\nc{\vin}{{\mathrm Vin}}    
\nc{\lin}{{\mathrm Lin}}    

\nc{\gop}{{\,\omega\,}}     
\nc{\gopb}{{\,\nu\,}}
\nc{\svec}[2]{{\tiny\left(\begin{matrix}#1\\
#2\end{matrix}\right)\,}}  
\nc{\ssvec}[2]{{\tiny\left(\begin{matrix}#1\\
#2\end{matrix}\right)\,}} 

\nc{\typeI}{local cocycle $3$-Lie bialgebra\xspace}
\nc{\typeIs}{local cocycle $3$-Lie bialgebras\xspace}
\nc{\typeII}{double construction $3$-Lie bialgebra\xspace}
\nc{\typeIIs}{double construction $3$-Lie bialgebras\xspace}

\nc{\bia}{{$\mathcal{P}$-bimodule ${\bf k}$-algebra}\xspace}
\nc{\bias}{{$\mathcal{P}$-bimodule ${\bf k}$-algebras}\xspace}

\nc{\rmi}{{\mathrm{I}}}
\nc{\rmii}{{\mathrm{II}}}
\nc{\rmiii}{{\mathrm{III}}}
\nc{\pr}{{\mathrm{pr}}}
\newcommand{\huaA}{\mathcal{A}}

\nc{\pll}{\beta}
\nc{\plc}{\epsilon}

\nc{\ass}{{\mathit{Ass}}}
\nc{\lie}{{\mathit{Lie}}}
\nc{\comm}{{\mathit{Comm}}}
\nc{\dend}{{\mathit{Dend}}}
\nc{\zinb}{{\mathit{Zinb}}}
\nc{\tdend}{{\mathit{TDend}}}
\nc{\prelie}{{\mathit{preLie}}}
\nc{\postlie}{{\mathit{PostLie}}}
\nc{\quado}{{\mathit{Quad}}}
\nc{\octo}{{\mathit{Octo}}}
\nc{\ldend}{{\mathit{ldend}}}
\nc{\lquad}{{\mathit{LQuad}}}

 \nc{\adec}{\check{;}} \nc{\aop}{\alpha}
\nc{\dftimes}{\widetilde{\otimes}} \nc{\dfl}{\succ} \nc{\dfr}{\prec}
\nc{\dfc}{\circ} \nc{\dfb}{\bullet} \nc{\dft}{\star}
\nc{\dfcf}{{\mathbf k}} \nc{\apr}{\ast} \nc{\spr}{\cdot}
\nc{\twopr}{\circ} \nc{\tspr}{\star} \nc{\sempr}{\ast}
\nc{\disp}[1]{\displaystyle{#1}}
\nc{\bin}[2]{ (_{\stackrel{\scs{#1}}{\scs{#2}}})}  
\nc{\binc}[2]{ \left (\!\! \begin{array}{c} \scs{#1}\\
    \scs{#2} \end{array}\!\! \right )}  
\nc{\bincc}[2]{  \left ( {\scs{#1} \atop
    \vspace{-.5cm}\scs{#2}} \right )}  
\nc{\sarray}[2]{\begin{array}{c}#1 \vspace{.1cm}\\ \hline
    \vspace{-.35cm} \\ #2 \end{array}}
\nc{\bs}{\bar{S}} \nc{\dcup}{\stackrel{\bullet}{\cup}}
\nc{\dbigcup}{\stackrel{\bullet}{\bigcup}} \nc{\etree}{\big |}
\nc{\la}{\longrightarrow} \nc{\fe}{\'{e}} \nc{\rar}{\rightarrow}
\nc{\dar}{\downarrow} \nc{\dap}[1]{\downarrow
\rlap{$\scriptstyle{#1}$}} \nc{\uap}[1]{\uparrow
\rlap{$\scriptstyle{#1}$}} \nc{\defeq}{\stackrel{\rm def}{=}}
\nc{\dis}[1]{\displaystyle{#1}} \nc{\dotcup}{\,
\displaystyle{\bigcup^\bullet}\ } \nc{\sdotcup}{\tiny{
\displaystyle{\bigcup^\bullet}\ }} \nc{\hcm}{\ \hat{,}\ }
\nc{\hcirc}{\hat{\circ}} \nc{\hts}{\hat{\shpr}}
\nc{\lts}{\stackrel{\leftarrow}{\shpr}}
\nc{\rts}{\stackrel{\rightarrow}{\shpr}} \nc{\lleft}{[}
\nc{\lright}{]} \nc{\uni}[1]{\tilde{#1}} \nc{\wor}[1]{\check{#1}}
\nc{\free}[1]{\bar{#1}} \nc{\den}[1]{\check{#1}} \nc{\lrpa}{\wr}
\nc{\curlyl}{\left \{ \begin{array}{c} {} \\ {} \end{array}
    \right .  \!\!\!\!\!\!\!}
\nc{\curlyr}{ \!\!\!\!\!\!\!
    \left . \begin{array}{c} {} \\ {} \end{array}
    \right \} }
\nc{\leaf}{\ell}       
\nc{\longmid}{\left | \begin{array}{c} {} \\ {} \end{array}
    \right . \!\!\!\!\!\!\!}
\nc{\ot}{\otimes} \nc{\sot}{{\scriptstyle{\ot}}}
\nc{\otm}{\overline{\ot}}
\nc{\ora}[1]{\stackrel{#1}{\rar}}
\nc{\ola}[1]{\stackrel{#1}{\la}}
\nc{\pltree}{\calt^\pl}
\nc{\epltree}{\calt^{\pl,\NC}}
\nc{\rbpltree}{\calt^r}
\nc{\scs}[1]{\scriptstyle{#1}} \nc{\mrm}[1]{{\rm #1}}
\nc{\dirlim}{\displaystyle{\lim_{\longrightarrow}}\,}
\nc{\invlim}{\displaystyle{\lim_{\longleftarrow}}\,}
\nc{\mvp}{\vspace{0.5cm}} \nc{\svp}{\vspace{2cm}}
\nc{\vp}{\vspace{8cm}} \nc{\proofbegin}{\noindent{\bf Proof: }}
\nc{\proofend}{$\blacksquare$ \vspace{0.5cm}}
\nc{\freerbpl}{{F^{\mathrm RBPL}}}
\nc{\sha}{{\mbox{\cyr X}}}  
\nc{\ncsha}{{\mbox{\cyr X}^{\mathrm NC}}} \nc{\ncshao}{{\mbox{\cyr
X}^{\mathrm NC,\,0}}}
\nc{\shpr}{\diamond}    
\nc{\shprm}{\overline{\diamond}}    
\nc{\shpro}{\diamond^0}    
\nc{\shprr}{\diamond^r}     
\nc{\shpra}{\overline{\diamond}^r}
\nc{\shpru}{\check{\diamond}} \nc{\catpr}{\diamond_l}
\nc{\rcatpr}{\diamond_r} \nc{\lapr}{\diamond_a}
\nc{\sqcupm}{\ot}
\nc{\lepr}{\diamond_e} \nc{\vep}{\varepsilon} \nc{\labs}{\mid\!}
\nc{\rabs}{\!\mid} \nc{\hsha}{\widehat{\sha}}
\nc{\lsha}{\stackrel{\leftarrow}{\sha}}
\nc{\rsha}{\stackrel{\rightarrow}{\sha}} \nc{\lc}{\lfloor}
\nc{\rc}{\rfloor}
\nc{\tpr}{\sqcup}
\nc{\nctpr}{\vee}
\nc{\plpr}{\star}
\nc{\rbplpr}{\bar{\plpr}}
\nc{\sqmon}[1]{\langle #1\rangle}
\nc{\forest}{\calf}
\nc{\altx}{\Lambda_X} \nc{\vecT}{\vec{T}} \nc{\onetree}{\bullet}
\nc{\Ao}{\check{A}}
\nc{\seta}{\underline{\Ao}}
\nc{\deltaa}{\overline{\delta}}
\nc{\trho}{\tilde{\rho}}

\nc{\rpr}{\circ}
\nc{\dpr}{{\tiny\diamond}}
\nc{\rprpm}{{\rpr}}

\nc{\mmbox}[1]{\mbox{\ #1\ }} \nc{\ann}{\mrm{ann}}
\nc{\Aut}{\mrm{Aut}} \nc{\can}{\mrm{can}}
\nc{\twoalg}{{two-sided algebra}\xspace}
\nc{\colim}{\mrm{colim}}
\nc{\Cont}{\mrm{Cont}} \nc{\rchar}{\mrm{char}}
\nc{\cok}{\mrm{coker}} \nc{\dtf}{{R-{\rm tf}}} \nc{\dtor}{{R-{\rm
tor}}}
\renewcommand{\det}{\mrm{det}}
\nc{\depth}{{\mrm d}}
\nc{\Div}{{\mrm Div}} \nc{\End}{\mrm{End}} \nc{\Ext}{\mrm{Ext}}
\nc{\Fil}{\mrm{Fil}} \nc{\Frob}{\mrm{Frob}} \nc{\Gal}{\mrm{Gal}}
\nc{\GL}{\mrm{GL}} \nc{\Hom}{\mrm{Hom}} \nc{\hsr}{\mrm{H}}
\nc{\hpol}{\mrm{HP}} \nc{\id}{\mrm{id}} \nc{\im}{\mrm{im}}
\nc{\incl}{\mrm{incl}} \nc{\length}{\mrm{length}}
\nc{\LR}{\mrm{LR}} \nc{\mchar}{\rm char} \nc{\NC}{\mrm{NC}}
\nc{\mpart}{\mrm{part}} \nc{\pl}{\mrm{PL}}
\nc{\ql}{{\QQ_\ell}} \nc{\qp}{{\QQ_p}}
\nc{\rank}{\mrm{rank}} \nc{\rba}{\rm{RBA }} \nc{\rbas}{\rm{RBAs }}
\nc{\rbpl}{\mrm{RBPL}}
\nc{\rbw}{\rm{RBW }} \nc{\rbws}{\rm{RBWs }} \nc{\rcot}{\mrm{cot}}
\nc{\rest}{\rm{controlled}\xspace}
\nc{\rdef}{\mrm{def}} \nc{\rdiv}{{\rm div}} \nc{\rtf}{{\rm tf}}
\nc{\rtor}{{\rm tor}} \nc{\res}{\mrm{res}} \nc{\SL}{\mrm{SL}}
\nc{\Spec}{\mrm{Spec}} \nc{\tor}{\mrm{tor}} \nc{\Tr}{\mrm{Tr}}
\nc{\mtr}{\mrm{sk}}

\nc{\ab}{\mathbf{Ab}} \nc{\Alg}{\mathbf{Alg}}
\nc{\Algo}{\mathbf{Alg}^0} \nc{\Bax}{\mathbf{Bax}}
\nc{\Baxo}{\mathbf{Bax}^0} \nc{\RB}{\mathbf{RB}}
\nc{\RBo}{\mathbf{RB}^0} \nc{\BRB}{\mathbf{RB}}
\nc{\Dend}{\mathbf{DD}} \nc{\bfk}{{\bf k}} \nc{\bfone}{{\bf 1}}
\nc{\base}[1]{{a_{#1}}} \nc{\detail}{\marginpar{\bf More detail}
    \noindent{\bf Need more detail!}
    \svp}
\nc{\Diff}{\mathbf{Diff}} \nc{\gap}{\marginpar{\bf
Incomplete}\noindent{\bf Incomplete!!}
    \svp}
\nc{\FMod}{\mathbf{FMod}} \nc{\mset}{\mathbf{MSet}}
\nc{\rb}{\mathrm{RB}} \nc{\Int}{\mathbf{Int}}
\nc{\Mon}{\mathbf{Mon}}
\nc{\remarks}{\noindent{\bf Remarks: }}
\nc{\OS}{\mathbf{OS}} 
\nc{\Rep}{\mathbf{Rep}}
\nc{\Rings}{\mathbf{Rings}} \nc{\Sets}{\mathbf{Sets}}
\nc{\DT}{\mathbf{DT}}

\nc{\BA}{{\mathbb A}} \nc{\CC}{{\mathbb C}} \nc{\DD}{{\mathbb D}}
\nc{\EE}{{\mathbb E}} \nc{\FF}{{\mathbb F}} \nc{\GG}{{\mathbb G}}
\nc{\HH}{{\mathbb H}} \nc{\LL}{{\mathbb L}} \nc{\NN}{{\mathbb N}}
\nc{\QQ}{{\mathbb Q}} \nc{\RR}{{\mathbb R}} \nc{\BS}{{\mathbb{S}}} \nc{\TT}{{\mathbb T}}
\nc{\VV}{{\mathbb V}} \nc{\ZZ}{{\mathbb Z}}


\nc{\calao}{{\mathcal A}} \nc{\cala}{{\mathcal A}}
\nc{\calc}{{\mathcal C}} \nc{\cald}{{\mathcal D}}
\nc{\cale}{{\mathcal E}} \nc{\calf}{{\mathcal F}}
\nc{\calfr}{{{\mathcal F}^{\,r}}} \nc{\calfo}{{\mathcal F}^0}
\nc{\calfro}{{\mathcal F}^{\,r,0}} \nc{\oF}{\overline{F}}
\nc{\calg}{{\mathcal G}} \nc{\calh}{{\mathcal H}}
\nc{\cali}{{\mathcal I}} \nc{\calj}{{\mathcal J}}
\nc{\call}{{\mathcal L}} \nc{\calm}{{\mathcal M}}
\nc{\caln}{{\mathcal N}} \nc{\calo}{{\mathcal O}}
\nc{\calp}{{\mathcal P}} \nc{\calq}{{\mathcal Q}} \nc{\calr}{{\mathcal R}}
\nc{\calt}{{\mathcal T}} \nc{\caltr}{{\mathcal T}^{\,r}}
\nc{\calu}{{\mathcal U}} \nc{\calv}{{\mathcal V}}
\nc{\calw}{{\mathcal W}} \nc{\calx}{{\mathcal X}}
\nc{\CA}{\mathcal{A}}

\nc{\fraka}{{\mathfrak a}} \nc{\frakB}{{\mathfrak B}}
\nc{\frakb}{{\mathfrak b}} \nc{\frakd}{{\mathfrak d}}
\nc{\oD}{\overline{D}}
\nc{\frakF}{{\mathfrak F}} \nc{\frakg}{{\mathfrak g}}
\nc{\frakm}{{\mathfrak m}} \nc{\frakM}{{\mathfrak M}}
\nc{\frakMo}{{\mathfrak M}^0} \nc{\frakp}{{\mathfrak p}}
\nc{\frakS}{{\mathfrak S}} \nc{\frakSo}{{\mathfrak S}^0}
\nc{\fraks}{{\mathfrak s}} \nc{\os}{\overline{\fraks}}
\nc{\frakT}{{\mathfrak T}}
\nc{\oT}{\overline{T}}
\nc{\frakX}{{\mathfrak X}} \nc{\frakXo}{{\mathfrak X}^0}
\nc{\frakx}{{\mathbf x}}
\nc{\frakTx}{\frakT}      
\nc{\frakTa}{\frakT^a}        
\nc{\frakTxo}{\frakTx^0}   
\nc{\caltao}{\calt^{a,0}}   
\nc{\ox}{\overline{\frakx}} \nc{\fraky}{{\mathfrak y}}
\nc{\frakz}{{\mathfrak z}} \nc{\oX}{\overline{X}}

\font\cyr=wncyr10

\nc{\redtext}[1]{\textcolor{red}{#1}}


\title{Classifying complements for conformal algebras}

\author{Yanyong Hong}
\address{Department of Mathematics, Hangzhou Normal University,
Hangzhou, 311121, P.R.China}
\email{hongyanyong2008@yahoo.com}
\thanks{This work was supported by the Zhejiang Provincial Natural Science Foundation of China (No. LY20A010022), the National Natural Science Foundation of China (No. 11871421) and the Scientific Research Foundation of Hangzhou Normal University (No. 2019QDL012)}
\subjclass[2010]{17B05, 17B65, 17B69}
\keywords{Lie conformal algebra, associative conformal algebra, bicrossed product, classifying complement}

\begin{abstract}
 Let $R\subseteq E$ be two Lie conformal algebras and $Q$ be a given complement of $R$ in $E$. Classifying complements problem asks for describing and classifying all complements of $R$ in $E$ up to an isomorphism. It is known that $E$ is isomorphic to a bicrossed product of $R$ and $Q$. We show that any complement of $R$ in $E$ is isomorphic to a deformation of $Q$ associated to the bicrossed product.
A classifying object is constructed to parameterize all $R$-complements of $E$. Several explicit examples are provided. Similarly, we also develop a classifying complements theory of associative conformal algebras.
\end{abstract}

\maketitle

\section{Introduction}
The notion of a Lie conformal algebra was introduced in \cite{K1, K2} providing an axiomatic description of
properties of the operator product expansion in conformal field theory. There are many other fields closely related to Lie conformal algebras such as vertex algebras, infinite-dimensional Lie algebras satisfying the locality property (\cite{K}) and Hamiltonian formalism in the theory of nonlinear
evolution equations (see \cite{BDK}). In particular, associative conformal algebras naturally appear in the representation theory of Lie conformal algebras.
Structure theory and representation theory of Lie conformal algebras and associative conformal algebras have been studied in a series of papers (see \cite{BKV}-\cite{DK1}, \cite{Ko1}-\cite{Ko2}, \cite{Z1}-\cite{Z2} and so on). In particular, the $\mathbb{C}[\partial]$-split extending structure problems of Lie conformal algebras and associative conformal algebras were investigated in \cite{HS} and \cite{H1} respectively. The tool for studying $\mathbb{C}[\partial]$-split extending structure problem is unified product, which includes crossed product, bicrossed product and so on. Note that bicrossed product is related to the $\mathbb{C}[\partial]$-split factorization problem, which asks for the description and classification of all Lie (or associative) conformal algebraic structures on the direct sum $E=R\oplus Q$ of $\mathbb{C}[\partial]$-modules, where $R$ and $Q$ are two given Lie (or associative) conformal algebras.

In this paper, following the works in \cite{HS} and \cite{H1}, we plan to study the following structure problem of Lie conformal algebras and associative conformal algebras:

{\bf Classifying complements problem}: Let $R\subseteq E$ be two Lie (or associative) conformal algebras. If
there exists an $R$-complement $Q$ of $E$, describe and classify up to an isomorphism all $R$-complements of $E$, i.e. all subalgebras $T$ of $E$, such that $E=R\oplus T$ as $\mathbb{C}[\partial]$-modules.\\
From the point of view of algebra, this problem is natural and significant. Similar problems for classical algebraic objects such as associative algebras, Hopf algebras, Lie algebras,  Leibniz algebras and left-symmetric algebras have been investigated in \cite{A, AM1, AM2, H1}.  According to the theory of bicrossed product developed in \cite{HS} and \cite{H1}, it is easy to see that $E$ in the classifying complements problem is just a bicrossed product of $R$ and $Q$ associated to a matched pair. Therefore, we can use bicrossed product theories to study  classifying complements problems of Lie conformal algebras and associative conformal algebras. It is shown that any complement $T$ of $R$ in $E$ is isomorphic to a deformation of $Q$ associated to the matched pair of the bicrossed product of $R$ and $Q$. Then a cohomological type object is constructed to parameterize all $R$-complements of $E$. Several explicit examples are also provided. In particular, in the case of Lie conformal algebras, it is shown that there exists finite Lie conformal algebras $R\subseteq E$ such that the number of isomorphism classes of complements of $R$ in $E$ is infinite. In addition, it should be pointed out that the method in this paper can be applied to other conformal algebras with a theory of bicrossed product such as left-symmetric conformal algebras and Leibniz conformal algebras.

This article is organized as follows. In Section 2, we introduce some
basic definitions of Lie conformal algebras, associative conformal algebras and complements. In Section 3, using bicrossed product construction, classifying complements problem of Lie conformal algebras is answered by a new cohomological type object which is associated to deformations of $Q$ related with the matched pair of bicrossed product of $R$ and $Q$. Several explicit examples are also provided. Section 4 is devoted to developing classifying complements theory of associative conformal algebras with a similar method in Section 3. 

Throughout this paper, denote by $\mathbb{C}$ the field of complex
numbers. All tensors over $\mathbb{C}$ are denoted by $\otimes$.
Moreover, if $A$ is a vector space, the space of polynomials of $\lambda$ with coefficients in $A$ is denoted by $A[\lambda]$.
\section{Preliminaries on  conformal algebras and complements}

In this section, we recall some basic definitions of Lie conformal algebras, associative conformal algebras and complements.

\begin{defi}\label{def1}{\rm
A {\bf conformal algebra} $R$ is a $\mathbb{C}[\partial]$-module
endowed with a $\mathbb{C}$-bilinear map $R\times R\rightarrow
 R[\lambda]$ denoted by $a\times b\rightarrow
a_{\lambda} b$ satisfying
\begin{eqnarray}
\partial a_{\lambda}b=-\lambda a_{\lambda}b,  \quad
a_{\lambda}\partial b=(\partial+\lambda)a_{\lambda}b.\end{eqnarray}

A {\bf Lie conformal algebra} $R$ is a conformal algebra with the $\mathbb{C}$-bilinear
map $[\cdot_\lambda \cdot]: R\times R\rightarrow  R[\lambda]$ satisfying
\begin{eqnarray*}
&&[a_\lambda b]=-[b_{-\lambda-\partial}a],~~~~\text{(skew-symmetry)}\\
&&[a_\lambda[b_\mu c]]=[[a_\lambda b]_{\lambda+\mu} c]+[b_\mu[a_\lambda c]],~~~~~~\text{(Jacobi identity)}
\end{eqnarray*}
for $a$, $b$, $c\in R$.

An {\bf associative conformal algebra} $R$ is a conformal algebra with the $\mathbb{C}$-bilinear
map $\cdot_\lambda \cdot: R\times R\rightarrow R[\lambda]$ satisfying
\begin{eqnarray}
(a_{\lambda}b)_{\lambda+\mu}c=a_{\lambda}(b_\mu
c),
\end{eqnarray}
for $a$, $b$, $c\in R$.}
\end{defi}

A  conformal algebra
is called {\bf finite} if it is finitely generated as a
$\mathbb{C}[\partial]$-module. The {\bf rank} of a conformal
algebra $R$ is its rank as a $\mathbb{C}[\partial]$-module.

\begin{ex}
The Virasoro Lie conformal algebra $\text{Vir}$ is the simplest nontrivial
example of Lie conformal algebras of rank one. It is defined by
$$\text{Vir}=\mathbb{C}[\partial]L, ~~[L_\lambda L]=(\partial+2\lambda)L.$$
\end{ex}

\begin{ex}
Let $\mathfrak{g}$ be a Lie (or associative) algebra. The current Lie (or associative) conformal
algebra associated to $\mathfrak{g}$ is defined by:
$$\text{Cur} \mathfrak{g}=\mathbb{C}[\partial]\otimes \mathfrak{g}, ~~[a_\lambda b]=[a,b] (\text{or $a_\lambda b=ab$}),
~~a,b\in \mathfrak{g}.$$
\end{ex}
Next, we introduce the definition of complement of conformal algebras.
\begin{defi}
Let $R\subseteq E$ be two conformal algebras. Another subalgebra $Q$ of $E$ is called an \emph{$R$-complement} of $E$, if $E=R\oplus Q$ as $\mathbb{C}[\partial]$-modules.
\end{defi}

\begin{defi}
Let $R\subseteq E$ be two conformal algebras. Denote by $\mathcal{F}(R,E)$ the isomorphism classes of $R$-complements of $E$. Define the \emph{factorization index} of $R$ in $E$ as
$[E: R]: =\mid \mathcal{F}(R,E)\mid$.
\end{defi}

Therefore, in the sequel, for investigating classifying complements problems, we only need to describe and compute $\mathcal{F}(R,E)$.

\section{Classifying complements for Lie conformal algebras}
In this section, we investigate classifying complements problem for Lie conformal algebras.

First, we introduce some definitions associated to classifying complements problem for Lie conformal algebras.
\begin{defi}
A \emph{left module} $M$ over a Lie conformal algebra $R$ is a $\mathbb{C}[\partial]$-module endowed with a $\mathbb{C}$-bilinear map
$R\times M\longrightarrow M[\lambda]$, $(a, v)\mapsto a_\lambda v$, satisfying the following axioms $(a, b\in R, v\in M)$:\\
(LM1)$\qquad\qquad (\partial a)_\lambda v=-\lambda a_\lambda v,~~~a_\lambda(\partial v)=(\partial+\lambda)a_\lambda v,$\\
(LM2)$\qquad\qquad [a_\lambda b]_{\lambda+\mu}v=a_\lambda(b_\mu v)-b_\mu(a_\lambda v).$
\end{defi}

Similarly, we can define  right $R$-module.
\begin{defi}
A \emph{right module} $M$ over a Lie conformal algebra $R$ is a $\mathbb{C}[\partial]$-module endowed with a $\mathbb{C}$-bilinear map
$M\times R\longrightarrow M[\lambda]$, $(v, a)\mapsto v_\lambda a$, satisfying the following axioms $(a, b\in R, v\in M)$:\\
(RM1)$\qquad\qquad(\partial v)_\lambda a=-\lambda v_\lambda a,~~v_\lambda (\partial a)=(\partial+\lambda)v_\lambda a,$\\
(RM2)$\qquad\qquad v_\mu[a_\lambda b]=(v_\mu a)_{\lambda+\mu}b-(v_\mu b)_{-\lambda-\partial}a.$
\end{defi}
Any right $R$-module is a left $R$-module via $a_\lambda v:=-v_{-\lambda-\partial}a$ and viceversa.

\begin{defi}
A \emph{matched pair} of Lie conformal algebras is a quadruple $(R,Q,\lhd_\lambda,\rhd_\lambda)$,
where $R$ and $Q$ are two Lie conformal algebras, $R$ is a left module of $Q$ under $\rhd_\lambda: Q\otimes R\rightarrow R$, $Q$ is a right module of $R$ under $\lhd_\lambda: Q\otimes R\rightarrow Q$ and they satisfy
\begin{eqnarray*}
(B1)&&~~x\rhd_{-\lambda-\mu-\partial}[a_\lambda b]=[(x\rhd_{-\lambda-\partial} a)_{-\mu-\partial} b]+[a_\lambda(x\rhd_{-\mu-\partial} b)]\\
&&+(x\lhd_{-\lambda-\partial} a)\rhd_{-\mu-\partial} b-(x\lhd_{-\mu-\partial} b)\rhd_{-\lambda-\partial}a,\\
(B2)&&~~\{x_\mu y\}\lhd_{-\lambda-\partial}a=\{x_\mu(y\lhd_{-\lambda-\partial}a)\}+\{(x\lhd_{-\lambda-\partial}a)_{\lambda+\mu}y\}\\
&&+x\lhd_{\mu}(y\rhd_{-\lambda-\partial}a)-y\lhd_{-\lambda-\mu-\partial}(x\rhd_{-\lambda-\partial}a),
\end{eqnarray*}
for any $a$, $b\in R$, $x$, $y\in Q$.
\end{defi}

Denote any element in $R\oplus Q$ by $a\oplus x$ for any $a\in R$ and $x\in Q$.

\begin{defi} (see Theorem 3.2 in \cite{HL})
Suppose that $(R,Q,\lhd_\lambda,\rhd_\lambda)$ is a matched pair of $R$ and $Q$. Then
the direct sum of $\mathbb{C}[\partial]$-modules  $R\oplus Q$ is a Lie conformal algebra endowed the $\lambda$-bracket as follows:
\begin{eqnarray}
[(a\oplus x)_\lambda (b\oplus y)]=([a_\lambda b]+x\rhd_\lambda b-y\rhd_{-\lambda-\partial} a)\oplus(\{x_\lambda y\}+x\lhd_\lambda b-y\lhd_{-\lambda-\partial} a),
\end{eqnarray}
for any $a$, $b\in R$, $x$, $y\in Q$. This Lie conformal algebra is called the \emph{bicrossed product} of $R$ and $Q$, and is denoted by $R\bowtie_{\lhd,\rhd}Q$.
\end{defi}

Note that in $E=R\bowtie_{\lhd,\rhd}Q$, $R$ and $Q$ are two subalgebras of $E$ and $R\cap Q=\{0\}$. Therefore, $Q$ is a complement of $R$ in $E$. Conversely, according to Proposition 5.7 in \cite{HS},
we have the following result.

\begin{pro}
Let $R$ and $E$ be two Lie conformal algebras and $R\subseteq E$. Set $Q$ be a complement of $R$ in $E$. Then $E$ is isomorphic to a bicrossed product $R\bowtie_{\lhd,\rhd}Q$ of Lie conformal algebras $R$ and $Q$
associated to a matched pair
$(R, Q, \lhd_\lambda, \rhd_\lambda)$.
\end{pro}

\begin{defi}
Suppose that $(R,Q,\lhd_\lambda,\rhd_\lambda)$ is a matched pair of Lie conformal algebras.
If a $\mathbb{C}[\partial]$-module homomorphism $\varphi: Q\rightarrow R$ satisfies
\begin{eqnarray}\label{eq1}
\varphi([x_\lambda y])-[\varphi(x)_\lambda \varphi(y)]
=\varphi(y\lhd_{-\lambda-\partial}\varphi(x))-\varphi(x\lhd_\lambda \varphi(y))
+x\rhd_\lambda\varphi(y)-y\rhd_{-\lambda-\partial}\varphi(x),
\end{eqnarray}
for any $x$, $y\in Q$, then $\varphi$ is called a \emph{deformation map} of  $(R,Q,\lhd_\lambda,\rhd_\lambda)$.
\end{defi}

\begin{rmk}
Suppose that $Q$ is an ideal of $E$. Then $\rhd_\lambda$ is trivial in the associated matched pair. We simply denote this matched pair by $(R,Q,\lhd_\lambda)$. In this case, any deformation map $\varphi$ satisfies
\begin{eqnarray}\label{eq2}
\varphi([x_\lambda y])-[\varphi(x)_\lambda \varphi(y)]
=\varphi(y\lhd_{-\lambda-\partial}\varphi(x))-\varphi(x\lhd_\lambda \varphi(y)),~~~\text{for any $x$, $y\in Q$}.
\end{eqnarray}

Note that if $\lhd_\lambda$ and $\rhd_\lambda$ are trivial in the matched pair $(R,Q,\lhd_\lambda,\rhd_\lambda)$, all deformation maps are homomorphisms of Lie conformal algebra.
\end{rmk}
We also denote the set of all deformation maps of $(R,Q,\lhd_\lambda,\rhd_\lambda)$ by
$\mathcal{DM}(Q,R\mid(\lhd_\lambda,\rhd_\lambda))$.

\begin{lem}\label{llem1}
Let $R$ and $E$ be two Lie conformal algebras, $R\subset E$, and $Q$ be a given complement of $R$ in $E$. Suppose $\varphi: Q\rightarrow R$ is a deformation map of the matched pair $(R,Q,\lhd_\lambda,\rhd_\lambda)$.\\
(1) Set $f_\varphi: Q\rightarrow E=R\oplus Q$ be a $\mathbb{C}[\partial]$-module homomorphism defined as follows:
\begin{eqnarray*}
f_\varphi(x)=\varphi(x)\oplus x,~~~\text{for any $x\in Q$.}
\end{eqnarray*}
Then $\overline{Q}=\text{Im}(f_\varphi)
$ is a complement of $R$ in $E$.\\
(2) Set $Q_\varphi=Q$ as a $\mathbb{C}[\partial]$-module. Then
$Q_\varphi$ is a Lie conformal algebra with the following $\lambda$-bracket
\begin{eqnarray}
[x_\lambda y]_\varphi:=[x_\lambda y]+x\lhd_\lambda \varphi(y)-y\lhd_{-\lambda-\partial}\varphi(x),~~~\text{for any $x$, $y\in Q$.}
\end{eqnarray}
$Q_\varphi$ is called the \emph{$\varphi$-deformation} of $Q$. Moreover, as Lie conformal algebras, $Q_\varphi\cong \overline{Q}$.
\end{lem}

\begin{proof}
(1) Obviously, $E=\overline{Q} \oplus R$ where the sum is the direct sum of $\mathbb{C}[\partial]$-modules. Therefore, we only need to prove that $\overline{Q}$ is a Lie conformal subalgebra. For any $x$, $y\in Q$,
\begin{eqnarray*}
&&[(\varphi(x)\oplus x)_\lambda (\varphi(y)\oplus y)]\\
&=&([\varphi(x)_\lambda \varphi(y)]+x\rhd_\lambda \varphi(y)
-y\rhd_{-\lambda-\partial}\varphi(x))\oplus([x_\lambda y]+x\lhd_\lambda \varphi(y)-y\lhd_{-\lambda-\partial}\varphi(x))\\
&=&(\varphi([x_\lambda y]+x\lhd_\lambda \varphi(y)-y\lhd_{-\lambda-\partial}\varphi(x)))\oplus
([x_\lambda y]+x\lhd_\lambda \varphi(y)-y\lhd_{-\lambda-\partial}\varphi(x)).
\end{eqnarray*}
Therefore, $\overline{Q}$ is a Lie conformal subalgebra of $E$. Then we finish the proof.

(2) Define a $\mathbb{C}[\partial]$-module homomorphism $g_\varphi: Q_\varphi\rightarrow \overline{Q}$
as $g_\varphi(x)=\varphi(x)\oplus x$ for any $x\in Q_\varphi$. Obviously, it is a
$\mathbb{C}[\partial]$-isomorphism. For proving (2), it is enough to show that $g_\varphi$ is a  Lie conformal algebra homomorphism. For any $x$, $y\in Q$,
\begin{eqnarray*}
g_\varphi([x_\lambda y]_\varphi)&=&g_\varphi([x_\lambda y]+x\lhd_\lambda \varphi(y)-y\lhd_{-\lambda-\partial}\varphi(x))\\
&=&(\varphi([x_\lambda y])+\varphi(x\lhd_\lambda \varphi(y))-\varphi(y\lhd_{-\lambda-\partial}\varphi(x)))\\
&&\oplus( [x_\lambda y]+x\lhd_\lambda \varphi(y)-y\lhd_{-\lambda-\partial}\varphi(x))\\
&=&([\varphi(x)_\lambda \varphi(y)]+x\rhd_\lambda \varphi(y)-y\rhd_{-\lambda-\partial} \varphi(x))\\
&&\oplus([x_\lambda y]+x\lhd_\lambda \varphi(y)-y\lhd_{-\lambda-\partial}\varphi(x))\\
&=&[(\varphi(x)\oplus x)_\lambda (\varphi(y)\oplus y)]\\
&=&[g_\varphi(x)_\lambda g_\varphi(y)].
\end{eqnarray*}
Therefore, $g_\varphi$ is a Lie conformal algebra homomorphism and the result holds.
\end{proof}
\begin{rmk}
Note that when $Q$ is an ideal of $E$, by (\ref{eq2}) , the deformation map $\varphi: Q\rightarrow R$ satisfies  $\varphi([x_\lambda y]_\varphi)=[\varphi(x)_\lambda \varphi(y)]$ for any $x$, $y\in Q$.
\end{rmk}
\begin{thm}\label{t1}
Let $R$ and $E$ be two Lie conformal algebras, $R\subseteq E$, $Q$ be a complement of $R$ in $E$
associated with the matched pair of Lie conformal algebras $(R, Q, \lhd_\lambda, \rhd_\lambda)$.
Then $\overline{Q}$ is a complement of $R$ in $E$ if and only if there exists a deformation map
$\varphi: Q\rightarrow R$ such that $\overline{Q}\cong Q_\varphi$.
\end{thm}
\begin{proof}
Suppose that $\overline{Q}$ is a complement of $R$ in $E$.
Then $E=R\oplus Q=R\oplus\overline{Q}$ as $\mathbb{C}[\partial]$-modules.
For any $x\in Q$, set $x=u(x)\oplus v(x)$ where $u: Q\rightarrow R$ and
$v: Q\rightarrow \overline{Q}$ are $\mathbb{C}[\partial]$-module homomorphisms.
Since the sum of $E=R\oplus\overline{Q}$ is the direct sum as $\mathbb{C}[\partial]$-modules,
it is easy to see that $v$ is injective. Similarly,
for any $y\in \overline{Q}$, set $y=f(y)\oplus g(y)$ where $f: \overline{Q}\rightarrow R$ and
$g: \overline{Q}\rightarrow Q$ are $\mathbb{C}[\partial]$-module homomorphisms.
Since $g(y)=u(g(y))\oplus v(g(y))=-f(y)\oplus y$, $vg=Id_{\overline{Q}}$.
So, $v$ is also surjective. Therefore, $v: Q\rightarrow \overline{Q}$ is a $\mathbb{C}[\partial]$-module isomorphism.

Define $\overline{v}: Q\rightarrow \overline{Q}\hookrightarrow E=R\bowtie_{\lhd,\rhd}\overline{Q}$,
as follows:
\begin{eqnarray*}
\overline{v}(x)=v(x)=-u(x)\oplus x,~~~~\text{for any $x\in Q$.}
\end{eqnarray*}
By Lemma \ref{llem1}, for proving this theorem, we only need to show that
$-u: Q\rightarrow R$ is a deformation map. This can be directly obtained with a similar computation as that in (1) of Lemma \ref{llem1} according to that
$\overline{Q}=\text{Im}(\overline{v})$ is a Lie conformal subalgebra of $E$.

Now the proof is finished.
\end{proof}
\begin{defi}
Let $(R, Q, \lhd_\lambda, \rhd_\lambda)$ be a matched pair of Lie conformal algebras.
For two deformation maps $\varphi$, $\phi: Q\rightarrow R$, if there exists
a $\mathbb{C}[\partial]$-module isomorphism $\alpha: Q\rightarrow Q$ such that
for any $x$, $y\in Q$,
\begin{gather}
\alpha([x_\lambda y])-[\alpha(x)_\lambda\alpha(y)]\nonumber\\
=\alpha(x)\lhd_\lambda \phi(\alpha(y))-\alpha(x\lhd_\lambda \varphi(y))-\alpha(y)\lhd_{-\lambda-\partial} \phi(\alpha(x))+\alpha(y\lhd_{-\lambda-\partial} \varphi(x)),
\end{gather}
then $\varphi$ and  $\phi$ are called \emph{equivalent} and denote this by
$\varphi\equiv \phi$.
\end{defi}
\begin{thm}\label{t2}
Let $R$ and $E$ be two Lie conformal algebras, $R\subseteq E$, and $Q$ be a complement of $R$ in $E$
associated with the matched pair of Lie conformal algebras $(R, Q, \lhd_\lambda, \rhd_\lambda)$. Then
$\equiv$ is an equivalence relation on $\mathcal{DM}(Q,R\mid(\lhd_\lambda,\rhd_\lambda))$ and
the map
\begin{eqnarray*}
\mathcal{HL}^2(Q,R\mid(\lhd_\lambda,\rhd_\lambda)):=\mathcal{DM}(Q,R\mid(\lhd_\lambda,\rhd_\lambda))/\equiv\rightarrow
\mathcal{F}(R,E),~~~~\overline{\varphi}\rightarrow Q_\varphi,
\end{eqnarray*}
is a bijection between $\mathcal{HL}^2(Q,R\mid(\lhd_\lambda,\rhd_\lambda))$ and
the isomorphism classes of complements of $R$ in $E$.
\end{thm}
\begin{proof}
It is easy to see that $Q_\phi\cong Q_\varphi$ if and only if $\phi$ and $\varphi$ are equivalent. Then this theorem follows by Theorem \ref{t1}.
\end{proof}
\begin{rmk}
By Theorem \ref{t2}, $[E: R]=\mid \mathcal{HL}^2(Q,R\mid(\lhd_\lambda,\rhd_\lambda)) \mid$.
Moreover, with the condition in Theorem \ref{t2}, when $R$ is an ideal of $E$, it is easy to see that $[E: R]=1$, i.e. all complements of $R$ in $E$ are isomorphic to $Q$.
\end{rmk}

\begin{rmk}
According to the characterization of free rank one Lie conformal algebras in \cite{K1}, any free rank one Lie conformal algebra is either abelian or isomorphic to Virasoro conformal algebra. Therefore,  with the condition in Theorem \ref{t2}, when $Q$ is a free $\mathbb{C}[\partial]$-module of rank one,  $[E: R]\leq 2$.
\end{rmk}
\begin{ex}\label{e1}
Let $\mathcal{W}(a,b)=\mathbb{C}[\partial]L\oplus \mathbb{C}[\partial]W$ be the Lie conformal algebra with $\lambda$-brackets given by
\begin{eqnarray}
[L_\lambda L]=(\partial+2\lambda)L,~~~[L_\lambda W]=(\partial+a\lambda+b)W,~~[W_\lambda W]=0,
\end{eqnarray}
for some $a$, $b\in \mathbb{C}$. Set $R=\mathbb{C}[\partial]L$ and $Q=\mathbb{C}[\partial]W$ be the subalgebras of $\mathcal{W}(a,b)$. Obviously, the abelian subalgebra $Q$ is a complement of $R$ in $\mathcal{W}(a,b)$. Moreover, $\mathcal{W}(a,b)=R\bowtie_{\lhd,\rhd}Q$ with $\rhd_\lambda$ trivial and $\lhd_\lambda$ satisfying $W\lhd_\lambda L=((a-1)\partial+a\lambda-b)W$.

Let $\varphi: Q\rightarrow R$ be a deformation map. Assume that $\varphi(W)=f(\partial)L$ for some $f(\partial)\in \mathbb{C}[\partial]$. Setting $x=W$ and $y=W$ in (\ref{eq2}) and comparing the coefficients of $L$, we can get
\begin{eqnarray}\label{eq3}
-f(-\lambda)f(\lambda+\partial)(\partial+2\lambda)=f(\partial)(-f(-\lambda)(\partial+a\lambda+b)-f(\lambda+\partial)((a-1)\partial+a\lambda-b)).
\end{eqnarray}
By comparing the coefficients of $\lambda$ in (\ref{eq3}), we can get $f(\partial)=\alpha$ for some $\alpha\in \mathbb{C}$. Taking it into (\ref{eq3}), one can get that $f(\partial)=0$ if $a\neq 1$, and
$f(\partial)=\alpha$ if $a=1$. When $a=1$, the corresponding Lie conformal algebra $Q_\alpha=Q_\varphi=\mathbb{C}[\partial]W$ is as follows:
\begin{eqnarray}
[W_\lambda W]=W\lhd_\lambda \varphi(W)-W\lhd_{-\lambda-\partial}\varphi(W)=\alpha (\partial+2\lambda)W.
\end{eqnarray}
Note that $Q_0$ is abelian and $Q_\alpha\cong Vir$ when $\alpha\neq 0$.
Therefore, by the discussion above, we can easily get that $[\mathcal{W}(a,b): R]=1$ when $a\neq 1$ and
$[\mathcal{W}(1,b): R]=2$.
\end{ex}
Then we generalize Example \ref{e1} as follows.
\begin{ex}
Let $E=\mathbb{C}[\partial]L\bigoplus\oplus _{i=1}^n\mathbb{C}[\partial]W_i$ be the Lie conformal algebra with $\lambda$-brackets as follows:
\begin{eqnarray}
[L_\lambda L]=(\partial+2\lambda)L,~~~[L_\lambda W_i]=(\partial+\lambda+b)W_i,~~[{W_i}_\lambda W_j]=0,
\end{eqnarray}
for any $i$, $j\in \{1,\cdots,n\}$ and $b\in \mathbb{C}$. Let $R=\mathbb{C}[\partial]L$ and
$Q=\oplus _{i=1}^n\mathbb{C}[\partial]W_i$ be two subalgebras of $E$. Obviously, $Q$ is a complement of $R$ in $E$ and  $E=R\bowtie_{\lhd,\rhd}Q$ with $\rhd_\lambda$ trivial and $\lhd_\lambda$ satisfying $W_i\lhd_\lambda L=(\lambda-b)W$.
For any deformation map $\varphi:Q\rightarrow R$, set $\varphi(W_i)=f_i(\partial)L$ for any $i\in \{1,\cdots, n\}$. Setting $x=W_i$ and $y=W_j$ in (\ref{eq2}), we can get
\begin{eqnarray}\label{eq4}
-f_i(-\lambda)f_j(\lambda+\partial)(\partial+2\lambda)=-f_i(-\lambda)f_i(\partial)(\partial+\lambda+b)
-f_i(\partial)f_j(\lambda+\partial)(\lambda-b).
\end{eqnarray}
Setting $i=j$ in (\ref{eq4}) and by the result of Example \ref{e1}, we get $f_i(\partial)=\alpha_i$ for any $i\in \{1,\cdots,n\}$ and $\alpha_i\in \mathbb{C}$. Then it is easy to see that (\ref{eq4}) naturally holds.
The $\lambda$-bracket on the $\varphi$-deformation of $Q$ is described as follows:
\begin{eqnarray}
[{W_i}_\lambda W_j]=\alpha_j(\lambda-b)W_i+\alpha_i(\partial+\lambda+b)W_j, ~~\text{for any $i$, $j\in \{1,\cdots, n\}$.}
\end{eqnarray}
Note that when $\alpha_i\neq 0$, $\mathbb{C}[\partial]W_i$ is isomorphic to the Virasoro conformal algebra. According to the number of Virasoro conformal algebra such as $\mathbb{C}[\partial]W_i$ and classification result of semisimple Lie conformal algebras in \cite{DK1}, we can get
$[E:R]=n$.
\end{ex}
\begin{ex}
Let $\widetilde{SV}=\mathbb{C}[\partial]L\oplus \mathbb{C}[\partial]Y\oplus \mathbb{C}[\partial]M \oplus \mathbb{C}[\partial]N$ be the extended Schr\"odinger-Virasoro Lie conformal algebra introduced in \cite{SY} with $\lambda$-brackets given by
\begin{eqnarray*}
&&[L_\lambda L]=(\partial+2\lambda)L,~~~[L_\lambda Y]=(\partial+\frac{3}{2}\lambda)Y,\\
&&[L_\lambda M]=(\partial+\lambda)M,~~~[Y_\lambda Y]=(\partial+2\lambda)M,\\
&&[L_\lambda N]=(\partial+\lambda)N,~~~[M_\lambda N]=-2M,\\
&&[Y_\lambda N]=-Y,~~~[Y_\lambda M]=[M_\lambda M]=[N_\lambda N]=0.
\end{eqnarray*}
Let $R=\mathbb{C}[\partial]L\oplus \mathbb{C}[\partial]N$ and $Q=\mathbb{C}[\partial]Y\oplus \mathbb{C}[\partial]M$ be two subalgebras of $\widetilde{SV}$. Moreover, $\widetilde{SV}=R\bowtie_{\lhd,\rhd}Q$ with $\rhd_\lambda$ trivial and
$\lhd_\lambda$ satisfying $M\lhd_\lambda L=\lambda M$, $Y\lhd_\lambda L=(\frac{1}{2}\partial+\frac{3}{2}\lambda)Y$, $M\lhd_\lambda N=- 2M$ and  $Y\lhd_\lambda N=-Y$.

Let $\varphi: Q\rightarrow R$ be a deformation map. Set $\varphi(Y)=f(\partial)L+g(\partial)N$ and $\varphi(M)=h(\partial)L+k(\partial)N$ for some $f(\partial)$, $g(\partial)$, $h(\partial)$ and $k(\partial)\in \mathbb{C}[\partial]$. Setting $(x,y)$ in (\ref{eq2}) be $(Y,Y)$, $(Y,M)$ and $(M,M)$ respectively and comparing the coefficients of $L$ and $N$, we can obtain
\begin{eqnarray}
&&\label{eqq1}(\partial+2\lambda)h(\partial)-f(-\lambda)f(\lambda+\partial)(\partial+2\lambda)
=f(\partial)(-f(-\lambda)(\partial+\frac{3}{2}\lambda)-g(-\lambda))\\
&&+f(\partial)(-f(\lambda+\partial)(\frac{1}{2}\partial+\frac{3}{2}\lambda)
+g(\lambda+\partial)),\nonumber\\
&&\label{eqq2}(\partial+2\lambda)k(\partial)-f(-\lambda)g(\lambda+\partial)(\partial+\lambda)-\lambda g(-\lambda)f(\lambda+\partial)
=g(\partial)(-f(-\lambda)(\partial+\frac{3}{2}\lambda))\\
&&+g(\partial)(-g(-\lambda)-f(\lambda+\partial)(\frac{1}{2}\partial+\frac{3}{2}\lambda)
+g(\lambda+\partial)),\nonumber\\
&&\label{eqq3}-f(-\lambda)h(\lambda+\partial)(\partial+2\lambda)=-(f(-\lambda)(\partial+\lambda)+2g(-\lambda))h(\partial)\\
&&-(h(\lambda+\partial)(\frac{1}{2}\partial+\frac{3}{2}\lambda)-k(\lambda+\partial))f(\partial),\nonumber\\
&&\label{eqq4}-f(-\lambda)k(\lambda+\partial)(\partial+\lambda)-g(-\lambda)h(\lambda+\partial)\lambda=-(f(-\lambda)(\partial+\lambda)+2g(-\lambda))k(\partial)\\
&&-(h(\lambda+\partial)(\frac{1}{2}\partial+\frac{3}{2}\lambda)-k(\lambda+\partial))g(\partial),\nonumber\\
&&\label{eqq5}-h(-\lambda)h(\lambda+\partial)(\partial+2\lambda)=h(\partial)(-(\partial+\lambda)h(-\lambda)-2k(-\lambda))\\
&&+h(\partial)(-h(\lambda+\partial)\lambda+2k(\lambda+\partial)),\nonumber\\
&&\label{eqq6}-h(-\lambda)k(\lambda+\partial)(\partial+\lambda)-k(-\lambda)h(\lambda+\partial)\lambda=k(\partial)(-(\partial+\lambda)h(-\lambda)-2k(-\lambda))\\
&&+k(\partial)(-h(\lambda+\partial)\lambda+2k(\lambda+\partial))\nonumber.
\end{eqnarray}
Setting $\lambda=0$ in (\ref{eqq6}), we can get $2k(\partial)(k(\partial)-k(0))=0$. Therefore,
$k(\partial)=d$ for some $d\in \mathbb{C}$. Taking it into (\ref{eqq5}) and comparing the degree of $\lambda$, one can obtain that $h(\partial)=c$ for some $c\in \mathbb{C}$. Then (\ref{eqq5}) and (\ref{eqq6}) naturally hold.

If $c\neq 0$, by (\ref{eqq3}), we can get
\begin{eqnarray}\label{eqq8}
-f(-\lambda)\lambda=-2g(-\lambda)-(\frac{1}{2}\partial+\frac{3}{2}\lambda)f(\partial)+\frac{d}{c}f(\partial).
\end{eqnarray}
By comparing the degree of $\partial$ in (\ref{eqq8}), one can obtain that $f(\partial)=g(\partial)=0$, which contradicts with (\ref{eqq1}). Therefore, $c=0$. It follows from (\ref{eqq3}) and (\ref{eqq4}) that $df(\partial)=0$ and
$d(g(\partial)-2g(-\lambda))=0$. Therefore, if $d\neq 0$, then $f(\partial)=g(\partial)=0$, which contradicts with (\ref{eqq2}). Thus, $d=0$. Then setting $\lambda=0$ in (\ref{eqq2}) yields
$g(\partial)(g(\partial)-\frac{1}{2}f(\partial)\partial-g(0))=0$. If $g(\partial)\neq 0$,
$g(\partial)=\frac{1}{2}f(\partial)\partial+g(0)$. Taking it into (\ref{eqq2}) and by some computation, we can get
\begin{eqnarray}\label{eqq7}
f(-\lambda)f(\lambda+\partial)(\partial^2+2\lambda\partial)=\lambda\partial f(\partial)f(\lambda+\partial)
+\lambda\partial f(\partial)f(-\lambda)+\partial^2f(\partial)f(-\lambda).
\end{eqnarray}
Therefore, it follows that $f(\partial)=a$ for some $a\in \mathbb{C}$ by comparing the degree of $\lambda$
in (\ref{eqq7}). Consequently, $g(\partial)=\frac{1}{2}a\partial+b$ for some $b\in \mathbb{C}$. Then (\ref{eqq1}) naturally holds. If $g(\partial)=0$, it can be obtained from (\ref{eqq1}) that $f(\partial)=0$.
By the discussion above, in this case,  we get that $h(\partial)=k(\partial)=0$, $f(\partial)=a$ and $g(\partial)=\frac{1}{2}a\partial+b$ for some $a$, $b\in \mathbb{C}$. Denote the corresponding deformation map by $\varphi_{a,b}$.

The $\varphi_{a,b}$-deformation $Q_{\varphi_{a,b}}$ of $Q$ can be described as follows:
\begin{eqnarray}
&&[Y_\lambda Y]_{\varphi_{a,b}}=(\partial+2\lambda)M+a(\partial+2\lambda)Y,\\
&&[Y_\lambda M]_{\varphi_{a,b}}=(a\partial+2b)M,~~[M_\lambda M]_{\varphi_{a,b}}=0.
\end{eqnarray}

Suppose $a\neq 0$. Note that $Q_{\varphi_{a,b}}$ is isomorphic to $\widetilde{Q}_{c}=\mathbb{C}[\partial]Y\oplus
\mathbb{C}[\partial]M$ with the following $\lambda$-brackets
\begin{eqnarray}
&&[Y_\lambda Y]=(\partial+2\lambda)Y+(\partial+2\lambda)M,\\
&&[Y_\lambda M]=(\partial+2c)M,~~[M_\lambda M]_{\varphi_{a,b}}=0,
\end{eqnarray}
through the isomorphism $\phi: Q_{\varphi_{a,b}}\rightarrow \widetilde{Q}_{c}$
given by $\phi(Y)=\frac{1}{a} Y$ and $\phi(M)= \frac{1}{a^2} M$ where $c=\frac{b}{a}$. Then it is easy to see that for any
$c_1\neq c_2 $, $\widetilde{Q}_{c_1}$ is not isomorphic to $\widetilde{Q}_{c_2}$.

When $a=0$, it is easy to see that $\varphi_{0,b}\equiv\varphi_{0,1}$ for all $b\neq 0$ and $\varphi_{0,1}$ is not equivalent to $\varphi_{0,0}$.

Note that when $a\neq 0$, $Q_{\varphi_{a,b}}$ is not solvable and $Q_{\varphi_{0,b}}$ is solvable. Therefore, $Q_{\varphi_{a,b}}$ is not isomorphic to $Q_{\varphi_{0,1}}$ and $Q_{\varphi_{0,0}}$. By the discussion above, the set $\mathcal{F}(R,E)$ can be described as $Q_{\varphi_{0,1}}$, $Q_{\varphi_{0,0}}$
and $\widetilde{Q}_{c}$ for all $c\in \mathbb{C}$. Consequently, $[E:R]=\infty$.
\end{ex}

\section{Classifying complements for associative conformal algebras}
In this section, we apply the method developed in Section 3 to the case of associative conformal algebras.

First ,we introduce the definitions of modules over associative conformal algebras.
\begin{defi}
A \emph{left module} $M$ over an associative conformal algebra $A$ is a $\mathbb{C}[\partial]$-module endowed with a $\mathbb{C}$-bilinear map
$A\times M\longrightarrow M[\lambda]$, $(a, v)\mapsto a\rightharpoonup_\lambda v$, satisfying the following axioms $(a, b\in A, v\in M)$:\\
(LM1)$\qquad\qquad (\partial a)\rightharpoonup_\lambda v=-\lambda a\rightharpoonup_\lambda v,~~~a\rightharpoonup_\lambda(\partial v)=(\partial+\lambda)a\rightharpoonup_\lambda v,$\\
(LM2)$\qquad\qquad (a_\lambda b)\rightharpoonup_{\lambda+\mu}v=a\rightharpoonup_\lambda(b\rightharpoonup_\mu v).$\\
We also denote it by $(M,\rightharpoonup_\lambda)$.

A \emph{right module} $M$ over an associative conformal algebra $A$ is a $\mathbb{C}[\partial]$-module endowed with a $\mathbb{C}$-bilinear map
$M\times A\longrightarrow M[\lambda]$, $(v, a)\mapsto v\lhd_\lambda a$, satisfying the following axioms $(a, b\in A, v\in M)$:\\
(RM1)$\qquad\qquad (\partial v)\lhd_\lambda a=-\lambda v\lhd_\lambda a,~~~v\lhd_\lambda(\partial a)=(\partial+\lambda)v\lhd_\lambda a,$\\
(RM2)$\qquad\qquad (v\lhd_\lambda a)_{\lambda+\mu}b=v\lhd_\lambda(a_\mu b).$\\
Usually, we denote it by $(M,\lhd_\lambda)$.

An $A$-\emph{bimodule} is $(M,\rightharpoonup_\lambda,\triangleleft_\lambda)$ such that $(M,\rightharpoonup_\lambda)$ is a left $A$-module, $(M,\lhd_\lambda)$ is a right $A$-module, and
they satisfy the following condition
\begin{eqnarray}
(a\rightharpoonup_\lambda v)\lhd_{\lambda+\mu}b=a\rightharpoonup_\lambda(v\lhd_\mu b),
\end{eqnarray}
where $a$, $b\in A$ and $v\in M$.
\end{defi}

\begin{defi}
Let $A$ and $Q$ be two associative conformal algebras.
A \emph{matched pair} of $A$ and $Q$ is of the form $(A,Q,\lhd_\lambda,  \rhd_\lambda, \leftharpoonup_\lambda,\rightharpoonup_\lambda)$, where
$(Q,\rightharpoonup_\lambda,\lhd_\lambda)$ is an $A$-bimodule, $(A,\rhd_\lambda, \leftharpoonup_\lambda)$ is a $Q$-bimodule, and
the following compatibility conditions hold for all $a$, $b\in A$, $x$, $y\in Q$:
\begin{eqnarray}
&&(a_\lambda b)\leftharpoonup_{\lambda+\mu}x=a_\lambda(b\leftharpoonup_\mu x)+a\leftharpoonup_\lambda(b\rightharpoonup_\mu x),\\
&&a_\lambda(x\rhd_\mu b)+a\leftharpoonup_\lambda(x\lhd_\mu b)
=(a\leftharpoonup_\lambda x)_{\lambda+\mu}b+(a\rightharpoonup_\lambda x)\rhd_{\lambda+\mu}b,\\
&&x\rhd_\lambda(a_\mu b)=(x\rhd_\lambda a)_{\lambda+\mu}b+(x\lhd_\lambda a)\rhd_{\lambda+\mu}b,\\
&&a\rightharpoonup_\lambda(x_\mu y)=(a\rightharpoonup_\lambda x)_{\lambda+\mu}y+(a\leftharpoonup_\lambda x)\rightharpoonup_{\lambda+\mu} y,\\
&&x\lhd_\lambda(a\leftharpoonup_\mu y)+x_\lambda(a\rightharpoonup_\mu y)
=(x\rhd_\lambda a)\rightharpoonup_{\lambda+\mu}y+(x\lhd_\lambda a)_{\lambda+\mu}y,\\
&&x\lhd_\lambda(y\rhd_\mu a)+x_\lambda(y\lhd_\mu a)
=(x_\lambda y)\lhd_{\lambda+\mu}a.
\end{eqnarray}
\end{defi}

\begin{defi} (see Theorem 3.2 in \cite{H})
Let $(A,Q,\lhd_\lambda,  \rhd_\lambda, \leftharpoonup_\lambda,\rightharpoonup_\lambda)$ be a matched pair of associative conformal algebras. Then the direct sum $A\oplus Q$ of $\mathbb{C}[\partial]$-modules is an associative conformal algebra with the following $\lambda$-product
\begin{eqnarray}
(a\oplus x)_\lambda (b\oplus y)=(a_\lambda b+a\leftharpoonup_\lambda y+x\rhd_\lambda b)\oplus (a\rightharpoonup_\lambda y+x\lhd_\lambda b+x_\lambda y),
\end{eqnarray}
for all $a$, $b\in A$ and $x$, $y\in Q$. This associative conformal algebra is called the \emph{bicrossed product} of $A$ and $Q$ and we denote it by  $A\bowtie Q$.
\end{defi}

\begin{pro} (see Proposition 4.4 in \cite{H})
Let $A$ and $E$ be two associative conformal algebras and $A\subseteq E$. Set $Q$ be a complement of $A$ in $E$. Then $E$ is isomorphic to a bicrossed product $A\bowtie Q$ of associative conformal algebras associated to a matched pair $(A,Q,\lhd_\lambda,  \rhd_\lambda, \leftharpoonup_\lambda,\rightharpoonup_\lambda)$.
\end{pro}

\begin{defi}
Let $(A,Q,\lhd_\lambda,  \rhd_\lambda, \leftharpoonup_\lambda,\rightharpoonup_\lambda)$ be a matched pair of associative conformal algebras. If a $\mathbb{C}[\partial]$-module homomorphism
$\varphi:Q\rightarrow A$ satisfies
\begin{eqnarray}
\varphi(x_\lambda y)-\varphi(x)_\lambda \varphi(y)
=\varphi(x)\leftharpoonup_\lambda y+x\rhd_\lambda\varphi(y)-\varphi(\varphi(x)\rightharpoonup_\lambda y)-\varphi(x\lhd_\lambda \varphi(y)),
\end{eqnarray}
for any $x$, $y\in Q$, then $\varphi$ is called a \emph{deformation map} of $(A,Q,\lhd_\lambda,  \rhd_\lambda, \leftharpoonup_\lambda,\rightharpoonup_\lambda)$.
\end{defi}

Denote the set of all deformation maps of $(A,Q,\lhd_\lambda,  \rhd_\lambda, \leftharpoonup_\lambda,\rightharpoonup_\lambda)$ by
$\mathcal{DM}(Q,A\mid(\lhd_\lambda,  \rhd_\lambda, \leftharpoonup_\lambda,\rightharpoonup_\lambda))$.

\begin{lem}\label{lem1}
Let $A$ and $E$ be two associative conformal algebras, $A\subset E$, and $Q$ be a given complement of $A$ in $E$. Suppose $\varphi: Q\rightarrow A$ is a deformation map of the matched pair $(A,Q,\lhd_\lambda,  \rhd_\lambda, \leftharpoonup_\lambda,\rightharpoonup_\lambda)$.\\
(1) Set $f_\varphi: Q\rightarrow E=A\oplus Q$ be a $\mathbb{C}[\partial]$-module homomorphism defined as follows:
\begin{eqnarray*}
f_\varphi(x)=\varphi(x)\oplus x,~~~\text{for any $x\in Q$.}
\end{eqnarray*}
Then $\overline{Q}=\text{Im}(f_\varphi)
$ is a complement of $A$ in $E$.\\
(2) Set $Q_\varphi=Q$ as a $\mathbb{C}[\partial]$-module. Then
$Q_\varphi$ is an associative conformal algebra with the following $\lambda$-product
\begin{eqnarray}
{x_\lambda y}_\varphi:=x_\lambda y+x\lhd_\lambda \varphi(y)+\varphi(x)\rightharpoonup_\lambda y,~~~\text{for any $x$, $y\in Q$.}
\end{eqnarray}
$Q_\varphi$ is called the \emph{$\varphi$-deformation} of $Q$. Moreover, as associative conformal algebras, $Q_\varphi\cong \overline{Q}$.
\end{lem}
\begin{proof}
The proof is similar to that in Lemma \ref{llem1}.
\end{proof}
\begin{thm}\label{t3}
Let $A$ and $E$ be two associative conformal algebras, $A\subseteq E$, $Q$ be a complement of $A$ in $E$
associated with the matched pair of associative conformal algebras $(A,Q,\lhd_\lambda,  \rhd_\lambda, \leftharpoonup_\lambda,\rightharpoonup_\lambda)$.
Then $\overline{Q}$ is a complement of $A$ in $E$ if and only if there exists a deformation map
$\varphi: Q\rightarrow A$ such that $\overline{Q}\cong Q_\varphi$.
\end{thm}
\begin{proof}
The proof is similar to that in Theorem \ref{t1}.
\end{proof}

\begin{defi}
Let $(A,Q,\lhd_\lambda,  \rhd_\lambda, \leftharpoonup_\lambda,\rightharpoonup_\lambda)$ be a matched pair of associative conformal algebras.
For two deformation maps $\varphi$, $\phi: Q\rightarrow A$, if there exists
a $\mathbb{C}[\partial]$-module isomorphism $\alpha: Q\rightarrow Q$ such that
for any $x$, $y\in Q$,
\begin{gather}
\alpha(x_\lambda y)-\alpha(x)_\lambda\alpha(y)\nonumber\\
=\alpha(x)\lhd_\lambda \phi(\alpha(y))-\alpha(x\lhd_\lambda \varphi(y))+\phi(\alpha(x))\rightharpoonup_\lambda \alpha(y)-\alpha( \varphi(x)\rightharpoonup_\lambda y),
\end{gather}
then $\varphi$ and  $\phi$ are called \emph{equivalent} and denote this by
$\phi\equiv \phi$.
\end{defi}
\begin{thm}\label{t4}
Let $A$ and $E$ be two associative conformal algebras, $A\subseteq E$, and $Q$ be a complement of $A$ in $E$
associated with the matched pair of associative conformal algebras $(A,Q,\lhd_\lambda,  \rhd_\lambda, \leftharpoonup_\lambda,\rightharpoonup_\lambda)$. Then
$\equiv$ is an equivalence relation on $\mathcal{DM}(Q,A\mid(\lhd_\lambda,  \rhd_\lambda, \leftharpoonup_\lambda,\rightharpoonup_\lambda))$ and
the map
\begin{eqnarray*}
\mathcal{HA}^2(Q,A\mid(\lhd_\lambda,  \rhd_\lambda, \leftharpoonup_\lambda,\rightharpoonup_\lambda)):=\mathcal{DM}(Q,A\mid(\lhd_\lambda,  \rhd_\lambda, \leftharpoonup_\lambda,\rightharpoonup_\lambda))/\equiv\rightarrow
\mathcal{F}(A,E),~~~~\overline{\varphi}\rightarrow Q_\varphi,
\end{eqnarray*}
is a bijection between $\mathcal{HA}^2(Q,A\mid(\lhd_\lambda,  \rhd_\lambda, \leftharpoonup_\lambda,\rightharpoonup_\lambda))$ and
the isomorphism classes of complements of $A$ in $E$.
\end{thm}
\begin{proof}
It is similar to that in Theorem \ref{t2}.
\end{proof}
\begin{rmk}
By Theorem \ref{t4}, $[E: A]=\mid \mathcal{HA}^2(Q,A\mid(\lhd_\lambda,  \rhd_\lambda, \leftharpoonup_\lambda,\rightharpoonup_\lambda)) \mid$.
Moreover, with the condition in Theorem \ref{t4}, when $A$ is an ideal of $E$, it is easy to see that $[E: A]=1$, i.e. all complements of $A$ in $E$ are isomorphic to $Q$.
\end{rmk}

\begin{ex}
Let $E=\mathbb{C}[\partial]e_1\oplus \mathbb{C}[\partial]e_2 \oplus \mathbb{C}[\partial]e_3
\oplus \mathbb{C}[\partial]e_4$ be an associative conformal algebra with the following nontrivial $\lambda$-products as follows:
\begin{eqnarray}
{e_1}_\lambda e_2= e_1,~~~{e_2}_\lambda e_2=e_2,~~{e_2}_\lambda e_3=e_3,~~{e_2}_\lambda e_4=e_4.
\end{eqnarray}
Set $A=\mathbb{C}[\partial]e_1\oplus \mathbb{C}[\partial]e_2$ and $Q=\mathbb{C}[\partial]e_3
\oplus \mathbb{C}[\partial]e_4$ be two subalgebras of $E$. It is easy to see that $E$ is the bicrossed product of $A$ and $Q$ with $\lhd_\lambda$, $\rhd_\lambda$ and $\leftharpoonup_\lambda$ trivial and
$\rightharpoonup_\lambda$ satisfying
\begin{eqnarray*}
{e_1}\rightharpoonup_\lambda e_3= {e_1}\rightharpoonup_\lambda e_4=0,~~{e_2}\rightharpoonup_\lambda e_3=e_3,~~{e_2}\rightharpoonup_\lambda e_4=e_4.
\end{eqnarray*}
For any deformation map $\varphi:Q\rightarrow A$, by the definition of deformation map and some computations, it is easy to see that $\varphi$ is of the following form:\\
(1) $\varphi(e_3)=f(\partial)e_1$, $\varphi(e_4)=h(\partial)e_1$ for some $f(\partial)$, $g(\partial)\in \mathbb{C}[\partial]$;\\
(2) $\varphi(e_3)=0$, $\varphi(e_4)=a e_1+b e_2$ where $a\in \mathbb{C}$, $b\in \mathbb{C}\setminus\{0\}$;\\
(3) $\varphi(e_3)=c e_1+d e_2$, $\varphi(e_4)=0$ where $c\in \mathbb{C}$, $d\in \mathbb{C}\setminus\{0\}$;\\
(4) $\varphi(e_3)=a e_1+b e_2$, $\varphi(e_4)=c e_1+d e_2$ where $a$, $b$, $c$, $d\in \mathbb{C}\setminus\{0\}$ and $ad=bc$.

Note that the $\varphi$-deformation of $Q$ corresponding to (1) is just $Q$.  The $\varphi$-deformation $Q_b$ of $Q$ corresponding to (2) can be described as follows:
\begin{eqnarray*}
{e_3}_\lambda e_3=0,~~{e_3}_\lambda e_4=0,~~{e_4}_\lambda e_3=b e_3,~~{e_4}_\lambda e_4=b e_4,
\end{eqnarray*}
where $b\in \mathbb{C}\setminus\{0\}$. Similarly,  the $\lambda$-products of the $\varphi$-deformation $\overline{Q}_d$ of $Q$ corresponding to (3) are as follows:
\begin{eqnarray*}
{e_3}_\lambda e_3=de_3,~~{e_3}_\lambda e_4=de_4,~~{e_4}_\lambda e_3=0,~~{e_4}_\lambda e_4=0,
\end{eqnarray*}
where $d\in \mathbb{C}\setminus\{0\}$. The $\varphi$-deformation $Q_{b,d}$ of $Q$ corresponding to (4) can be described as follows:
\begin{eqnarray*}
{e_3}_\lambda e_3=be_3,~~{e_3}_\lambda e_4=be_4,~~{e_4}_\lambda e_3=d e_3,~~{e_4}_\lambda e_4=d e_4,
\end{eqnarray*}
where $b$, $d\in \mathbb{C}\setminus\{0\}$. Obviously, $Q_b\cong \overline{Q}_d\cong Q_1$ and $Q_{b,d}\cong Q_{1,1}$. Moreover, $Q$, $Q_1$ and $Q_{1,1}$ are not isomorphic to each other. Therefore, the set $\mathcal{F}(A,E)$ can be descried as $Q$, $Q_1$ and $Q_{1,1}$. Consequently, $[E:A]=3$.
\end{ex}

\end{document}